\documentclass[a4paper,10pt]{article}
\usepackage{amsmath,amssymb,amsthm,latexsym,graphicx}

\theoremstyle{plain}
\newtheorem{lem}{Lemma}

\newtheorem{thm}[lem]{Theorem}

\theoremstyle{definition}

\newtheorem{amp}{Example}
\theoremstyle{remark}
\newtheorem{rem}[lem]{Remark}

\def\P{\mathbb P}
\def\Y{\mathbb Y}

\def\Q{\mathbb Q}

\title{Algorithms for Del Pezzo Surfaces of Degree 5\\
	(Construction, Parametrization)}

\author{Jon Gonz\'alez--S\'anchez\thanks{Departamento de Matem\'aticas, Universidad Aut\'onoma de Madrid, 28049 Madrid, Spain.} 
and Michael Harrison\thanks{RICAM, Austrian Academy of Sciences, 4040 Linz, Austria} \\ 
and Irene Polo--Blanco\thanks{Departamento de Matem\'aticas, Estad\'istica y Computaci\'on, Universidad de Cantabria, 39005 Santander, Spain} 
and Josef Schicho$^\dagger$
}
\begin{document}

\maketitle

\begin{abstract}
It is well known that every Del Pezzo surface of degree $5$ defined over a field $k$ is parametrizable over $k$. In this paper we give an algorithm for parametrizing, as well as algorithms for constructing examples in every isomorphism class and for deciding equivalence.
\end{abstract}

\section*{Introduction}

It is well-known that every Del Pezzo surface of degree~5 over a field $k$
(not necessarily algebraically closed) has a proper parametrization
with coefficients in $k$ (see \cite{Swinnerton:72,Sheperd-Barron:92}).
In this paper, we give a simple algorithm for constructing such a parametrization.
The construction is a slight modication of Sheperd-Barrons construction (the
modification is important for getting a good performance). In addition,
we give a simple algorithm for constructing example of Del Pezzo
surfaces of degree~5. The algorithm takes as input a quintic squarefree
univariate polynomial over $k$. The construction is complete in the sense
that any Del Pezzo surface of degree~5 is obtained up to a
projective coordinate change with matrix entries in the ground field.

The classification of quintic Del Pezzo surfaces up to projective isomorphisms
defined over $k$ is also well-known (see \cite{Skorobogatov:01}, Lemma 3.1.7). 
Any such surface has 10 lines defined over the algebraic closure $\bar{k}$,
and the incidence graph of the configuration of these lines has symmetry group $S_5$.
There is a Galois action on the set of lines with image in $S_5$.
Two quintic Del Pezzo surfaces are isomorphic if the two Galois actions on
the line configurations are $S_5$-conjugate. The Galois action
may also be described as the Galois action on the roots of the quintic
input polynomial in our constuction algorithm. We use this explicit
decription to give an algorithm for deciding isomorphy of two given
quintic Del Pezzo surfaces.

The parametrization algorithm was motivated by the more general problem of computing
rational parametrizations over $k$ for arbitrary rational surfaces (if possible).
By Enriques-Manin reduction (see \cite{Iskovskih:80} for the theory and \cite{Schicho:97} 
for an algorithm), one can birationally reduce either to a conical fibration or to 
a Del Pezzo surface. For conic fibration and for Del Pezzo surfaces of degree 6, 8, or 9, 
parametrization algorithms are available (\cite{Schicho:00d,Schicho:06d,Schicho:06f,Schicho:08b})
and implemented in Magma~\cite{MAGMA}. For all other degrees except 5
the surfaces are either not properly parametrizable with coefficients
in the ground field, or they can be reduced to degree 6, 8, or 9
(see also \cite{Polo_Top:08,Swinnerton:70,Iskovskih:80}).

The first author acknowledges support by the Spanish Ministerio de Ciencia e Innovaci\'on, grant MTM2008-06680-C02-01.
The first and third author were partially supported by the Marie-Curie Initial Training Network (FP7-PEOPLE-2007-1-1-ITN) SAGA (ShApes, Geometry and Algebra).
The fourth author was partially supported by the Austrian Science Fund (FWF), project 21461-N23.

\section{Theory}

Throughout, we assume that $k$ is a perfect field and $\bar{k}$ is an algebraic
closure of $k$. We are primarily interested in the case $k=\Q$, but
all constructions represented work also in the general case. Projective
algebraic varities are defined as subsets of projective space
over $\bar{k}$, but we assume that all varieties are defined by
equations with coefficients in $k$, and consequently all constructions
will be possible within $k$.

Abstractly, a Del Pezzo surface is defined as a complete nonsingular surface
such that the anticanonical divisor $-K$ is ample. The integer $d:=K^2$ is called
the degree of the Del Pezzo surface. If $d\ge 3$, then $-K$ is very ample
and defines a natural embedding in $\P^d$ as a surface of degree~$d$.
Conversely, it is known that every non-singular surface of degree $d$ in $\P^d$
is either Del Pezzo or ruled or the projection of the Veronese surface of degree~4
in $\P^5$ to $\P^4$.

The general theory of Del Pezzo surfaces which is relevant to this paper
may be summarized by the following well-known theorems.

\begin{thm} \label{thm:general} 
If $F\subseteq\P^d$ is a Del Pezzo surface of degree $d$, then $3\le d\le 9$.

If $d\ne 8$, then $F$ is 
$\bar{k}$-isomorphic to the
blowing up of $\P^2$ at $9-d$ points in general position, i.e. no 3
points lie on a line and no 6 points on a conic. 

If $d=8$, then $S$ is
$\bar{k}$-isomorphic to either the blowup of $\P^2$ at a point or to 
$\P^1\times\P^1$.
\end{thm}

\begin{proof}
See \cite[Chap. IV Theorem 24.3 and Theorem 24.4]{Manin:74}. 
\end{proof}

\begin{thm} \label{thm:gpar}
Let $F\subseteq\P^d$ be a Del Pezzo surface of degree $d$, $d\ne 8$.
Then there exists a birational parametrization $\phi:\P^2\to F$,
$p\mapsto(P_0(p):\dots:P_d(p))$, such that $(P_0,\dots,P_d)$ 
are a basis for the vectorspace of all cubics vanishing at
$9-d$ points in general position. This parametrization
has coefficients in $\bar{k}$.
\end{thm}

\begin{proof}
See \cite[Chap. IV, proof of Theorem 24.5]{Manin:74} or \cite[Corollary 2]{Schicho:05}.
\end{proof}

Let $F\subseteq\P^5$ be a Del Pezzo surface of degree~5.
By Theorem~\ref{thm:gpar}, the surface can be parametrized
by cubics vanishing at four base points $q_1,\dots,q_4$. 
The surface $F$ contains 4 exceptional lines $E_1,\dots,E_4$, 
which are the preimages of $q_1,\dots,q_4$
under the inverse of the parametrization $\phi:\P^2\to F$.
For any two distinct base points $q_i,q_j$, the image of the
line connecting $q_i$ and $q_j$ under $\phi$ is also a line,
which we denote by $L_{ij}$. These lines are all lines on $F$.
Drawing a vertex for every line and an edge between vertices
such that the corresponding lines meet, we obtain the 
Petersen graph \includegraphics[width=1.5cm,height=1.5cm]{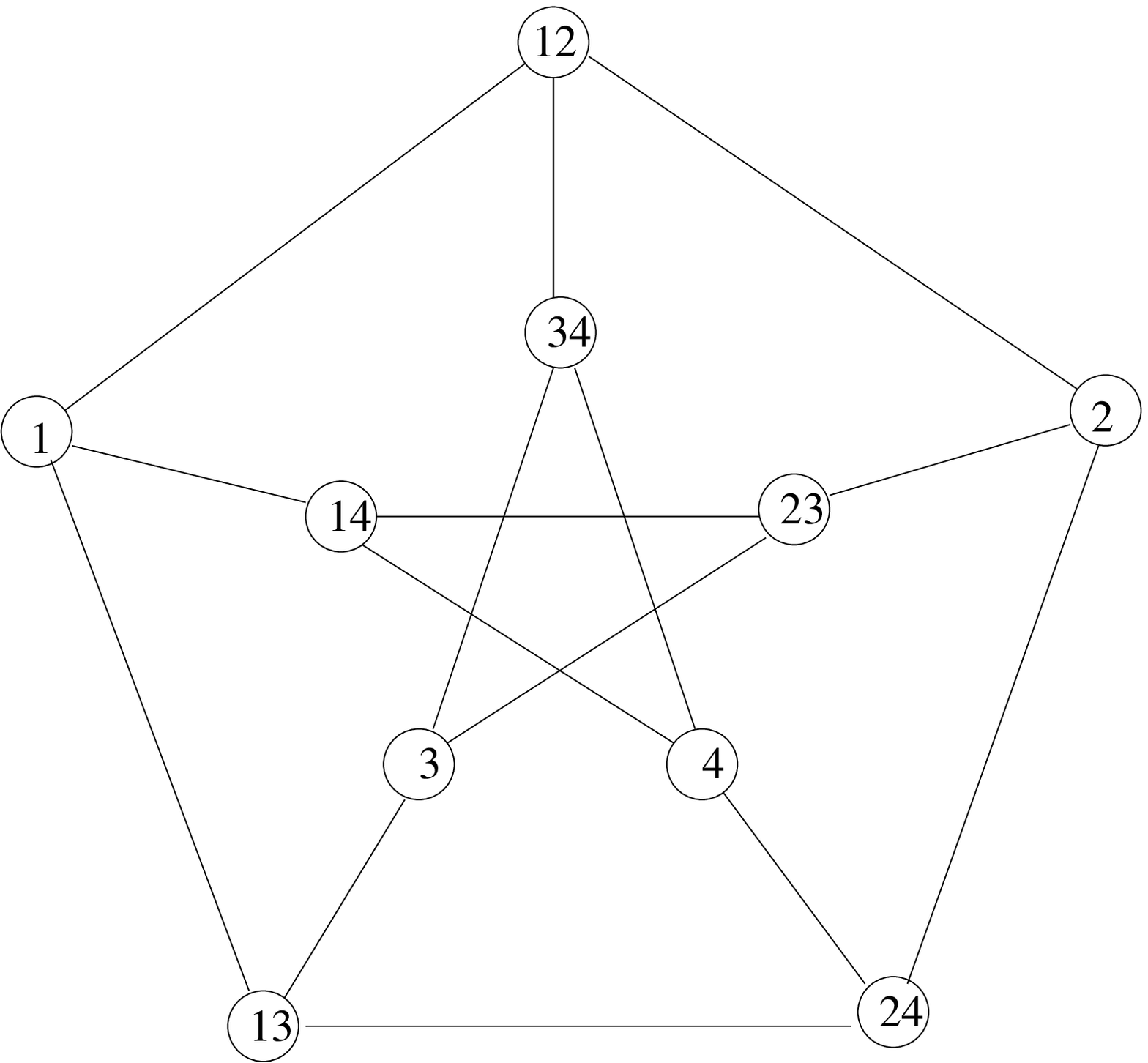}.

\begin{thm} \label{thm:res}
The ideal $I$ of a Del Pezzo surface $F\subseteq\P^5$ of degree~5
is generated by 5 quadrics $P_1,\dots,P_5\in R:=k[x_0,\dots,x_5]$.
The syzygy module
\[ \mathrm{Syz}_{(P_1,\dots,P_5)}= \{ (A_1,\dots,A_5)\in R^5\mid 
	A_1P_1+\dots+A_5P_5=0 \} \]
is generated by 5 vectors $V_1,\dots,V_5\in R^5$ of linear forms.
The second syzygy module
\[ \mathrm{Syz}_{(V_1,\dots,V_5)}= \{ (B_1,\dots,B_5)\in R^5\mid 
	B_1V_1+\dots+B_5V_5=0 \} \]
is generated by a single vector $W=(W_1,\dots,W_5)\in R^5$ of quadrics.
The entries $W_i$, $i=1,\dots,5$ generate the ideal $I$. 
With a suitable choice of basis $(V_1,\dots,V_5)$ of the linear part of 
$\mathrm{Syz}_{(P_1,\dots,P_5)}$, one can achieve that $W_i=B_i$ for
$i=1,\dots,5$i. then the matrix $M:=(V_{ij})_{i,j}$ of linear forms 
is skew symmetric, and the ideal is generated by the 5 first
Pfaffian minors of $M$.
\end{thm}

\begin{proof}
See Theorem 2.2 in \cite{Brodmann_Schenzel:06}.
\end{proof}


\begin{amp} \label{amp:std}
Consider $p_1=(1:0:0)$, $p_2=(0:1:0)$, $p_3=(0:0:1)$ and $p_4=(1:1:1)$. 
Let $V$ be the space of cubics in $\Q [t_0,t_1,t_2]$
vanishing at $p_1$, $p_2$, $p_3$ and $p_4$. The following is a basis for $V$:
\begin{eqnarray*}
&P_1=t_1^2t_2-t_1t_2t_0, \ \
P_2=t_1^2t_0-t_1t_2t_0, \ \ 
P_3=t_0^2t_1-t_1t_2t_0, \\
&P_4=t_0^2t_2-t_1t_2t_0, \ \ 
P_5=t_2^2t_0-t_1t_2t_0, \ \
P_6=t_2^2t_1-t_1t_2t_0.
\end{eqnarray*}
The map $p\to (P_0(p):P_1(p):P_2(p):P_3(p):P_4(p):P_5(p))$ defines a parametrization of a 
Del Pezzo suface $S$ of degree $5$ which is isomorphic to 
the blow up of $\P^2$ at the points $p_1$, $p_2$, $p_3$ and $p_4$. With the matrix
\[ M =  \begin{pmatrix} 0 & -x_0+x_1 & -x_1 & x_1-x_5 & x_5 \\
	x_0-x_1 & 0 & -x_2 & -x_5 & x_5 \\
	x_1 & x_2 & 0 & x_2 & -x_3 \\
	-x_1+x_5 & x_5 & -x_2 & 0 & x_4 \\
	-x_5 & -x_5 & x_3 & -x_4 & 0 \end{pmatrix}  , \]
the surface $S$ is generated by the 5 Pfaffians of the $4\times 4$ diagonally symmetric
submatrices, and the syzigies are generated by the columns (or rows) of $M$.

\end{amp}


\section{Construction} \label{sec:cons}

We are going to describe a simple algorithm that takes as input a quintic normed 
and squarefree univariate polynomial $Q\in k[x]$, called the {\em seed},
and produces a quintic Del Pezzo surface. Roughly speaking, the seed is
used to construct 5 points in $\P^2$ in general position (defined over $\bar{k}$),
and the surface is the image of $\P^2$ under the map defined by quintics
vanishing doubly at the 5 points.

\paragraph{First step.} 
Using the seed, we construct a zero-dimensional
subvariety $B\subset\P^2$ defined over $k$, which is
the set of 5 points $\{q_1,\dots,q_5\}$ defined over $\bar{k}$ 
in general position. This means, no three points are collinear.
Geometrically, $B$ is the image of the zeroes of $Q$ under the
map $\psi:\bar{k}\to\P^2$, $x\mapsto (x^2:x:1)$. 
Algebraically, we get the ideal of $B$
as the quotient of the ideal $I_1$ generated by the homogenization
$Q_h(t,u)$ and $t^2-su$ by the ideal $I_2$ generated by $u$.

\begin{lem}
Any set $B\subset\P^2$ of 5 points defined over $k$ in general position
can be obtained by the above construction, up to a projective
coordinate change of $\P^2$ with matrix entries in $k$.
\end{lem}

\begin{proof}
For any such set $B$, there is a unique conic $C_1$ that contains $B$.
Since $B$ is defined over $k$, $C_1$ is also defined over $k$.
The conic $C_1$ contains also a $k$-rational point: it can be constructed
by choosing generically a cubic $C_2$ through $B$ and computing the 6-th
intersection point of $C_1$ and $C_2$. Any two conics in $\P^2$ defined
over $k$ with $k$-rational points are projectively equivalent by
a projective map with coefficients in $k$, so we may transform
the conic $C_1$ to the conic $C$ with equation $s^2-tu$. We may
assume that the point at infinity $(1:0:0)$ is not in $B$;
this can be achieved by applying an element of the transitive group 
of projective automorphisms of $\bar{k}$.
Then $\psi^{-1}(B)$ is a subset of five points in $\bar{k}$ 
which is invariant under the Galois group $\mathrm{Gal}(\bar{k}/k)$, 
and so there is a quintic normed squarefree polynomial $Q\in k[x]$ 
that has these points as zeroes.
\end{proof}

\paragraph{Second step.}
We construct a rational map $\phi:\P^2\dashrightarrow\P^5$;
we will show that this map is birational, and that the image $F$ is
a Del Pezzo surface of degree~5. The map is defined by
the vector space of quintics vanishing with order
at least~2 at the points in $B$. In terms of computations,
we compute the square of the ideal of $B$, saturate it with
respect to the irrelevant ideal, and take a basis of the homogeneous
part of degree~5.

\begin{thm}
Let $q_1, ... ,q_5$ be five points on $\P^2$ in general position and
let $\phi$ be the rational map from $\P^2$ into a projective space defined 
by the space of quintics passing through $q_1,\dots,q_5$
with order at least $2$. Then $\phi$ maps $\P^2$ into
a del Pezzo surface $F$ of degree $5$ in $\P^5$. 

The 10 lines on $F$ are the strict transforms of the lines
connecting two of the points $q_i$ and $q_j$, for $1\leq ij,\leq 5$.
\end{thm}


\begin{proof}
It is well known that a linear system of quintics with 5 double points
has dimension~5 and has no unassigned base points (see \cite{Nagata:60}).
It follows that $\phi$ is regular and maps into $\P^5$.

Because 
the self-intersection number of the linear system is 5, and this
is equal to the degree of the image times the mapping degree, we
also conclude that $\phi$ is birational and the image is a surface $F$
of degree~5.

To show that the image is nonsingular, we start by resolving the 5 base
points. Let $\Y$ be the blowup of these 5 points. Then $\phi$ induces
a regular map from $\Y$ to the image surface $F$, associated to a divisor class $H$.
We can write $H=5L-2E$, where $L$ is the pullback of the class of lines
and $E=E_1+\dots+E_5$ is the class of the exceptional divisor of the blowup consisting
of 5 components. Note that $L^2=1$, $LE=0$ and $E^2=-5$.
The only curve on $\Y$ which is contracted to a point is the proper transform
of the conic $C$, with class $2L-E$, because this is the only curve with 
intersection number $0$ with $H$. This is a -1-curve, hence the image is nonsingular
by Theorem~21.5 of \cite{Manin:74}.

Finally, the lines on $F$ are precisely the curves $D$ on $\Y$ 
such that $D\cdot H=1$. This is only possible if $D$ is the proper
transform of a line through two points in $B$, with class $L-E_i-E_j$,
$1\le i<j\le 5$. These are only
finitely many. Then the image is not a ruled surface, therefore
it is a Del Pezzo surface.
\end{proof}

\begin{amp} \label{amp:example2}
Consider $Q=x^5-1$ and let $p_1=(\zeta_5^2:\zeta_5:1)$, $p_2=(\zeta_5^4:\zeta_5^2:1)$, $p_3=(\zeta_5:\zeta_5^3:1)$, $p_4=(\zeta_5^3:\zeta_5^4:1)$ and $p_5=(1:1:1)$ where $\zeta_5$ is a primitive $5$-root of unity. 
Let $V$ be the space of quintics in $\Q [t_0,t_1,t_2]$
vanishing with multiplicity $2$ at $p_1$, $p_2$, $p_3$, $p_4$ and $p_5$. A base of $V$ is the following
\begin{align*}
&P_1=t_0^5 - 5t_0t_1^2t_2^2 + 2t_1^5 + 2t_2^5,
&P_2=    t_0^4t_1 - 2t_0^2t_1^2t_2 + t_1^3t_2^2,\ \ \ \ \ \ \ &\\ 
&P_3=    t_0^4t_2 - 2t_0^2t_1t_2^2 + t_1^2t_2^3,
&P_4=    t_0^3t_1^2 - t_0^2t_2^3 - t_0t_1^3t_2 + t_1t_2^4,&\\ 
&P_5=    t_0^3t_1t_2 - 3t_0t_1^2t_2^2 + t_1^5 + t_2^5,
&P_6=    t_0^3t_2^2 - t_0^2t_1^3 - t_0t_1t_2^3 + t_1^4t_2.&
\end{align*}
The map $p\to (P_1(p):P_2(p):P_3(p):P_4(p):P_5(p):P_6(p))$ defines a parametrization of a 
Del Pezzo suface $S$ of degree $5$. The Del Pezzo 
surface $S$ is defined by the five  quadrics in $\Q [x_0,x_1,x_2,x_3,x_4,x_5]$:
\begin{eqnarray*}
&   x_1x_5 - x_2x_4 + x_3^2,\\
&    x_1x_4 - x_2x_3 - x_5^2,\\
&    x_0x_5 + x_1x_3 - x_2^2 - 2x_4x_5,\\
&    x_0x_4 - x_1x_2 + x_3x_5 - 2x_4^2,\\
&    x_0x_3 - x_1^2 + x_2x_5 - 2x_3x_4.
\end{eqnarray*}
\end{amp}

In section~\ref{sec:complete}, we will show that every quintic Del
Pezzo surface is isomorphic to a surface constucted as above with
some suitable seed. However, it is not true that different choices of the seed lead to
non-isomorphic surfaces. For instance, if $Q$ is a product of linear
factors, then the surface $F$ is $k$-isomorphic to Example~1, for
all choices of the linear factors. See section~\ref{sec:recog} for
more details.

\section{Parametrization} \label{sec:par}

In this section, we describe an algorithm that takes as an input the
defining equations of a quintic Del Pezzo surface $F$ -- these are  5 quadratic
equations in 6 variables with coefficients in $k$ --, and produces
a proper rational parametrization, i.e. a birational map from $\P^2$ to $F$
defined over $k$.

\paragraph{First Step.}
We construct a point on $F$ with coordinates in $k$. 
The first substep of this step is
to construct a point on $F$ with coordinates in $k$ or in a 
quadratic extension of $k$.

Let $R:=k[x_0,\dots,x_5]$ be the graded coordinate ring of $\P^5$.
The ideal of $F$ is generated by 5 quadratic equations, 
by Theorem~\ref{thm:res}. By the same theorem, we also know the
degree and the number of syzygies of the ideal.

\begin{lem}
Let $P_1,\dots,P_5\in R$ be generators of the ideal of $F$.
Let $V_1,\dots,V_5$ be generators of the syzygy module
$\mathrm{Syz}_{(P_1,\dots,P_5)}$. For $i=1,\dots,5$, let
$V_5=(V_{5,1},\dots,V_{5,5})$ be the last syzygy vector.
Then the zero set $L$ of $(V_{1,5},\dots,V_{5,5})$ is a linear
subspace of $\P^5$ of dimension~1 or 2. Moreover, the intersection
of $L$ and the zero set of $P_5$ is contained in $F$.
\end{lem}

\begin{proof}
Let $I_1:=\langle V_{1,5},\dots,V_{5,5}\rangle_R$,
$I_2:=\langle P_1,\dots,P_4\rangle_R$, and $I_3:=\langle P_5\rangle_R$.
Then we have $I_1=I_2:I_3$ and it follows $I_2\subseteq I_1$
and $I_2+I_3\subseteq I_1+I_3$, hence the common zero set of $I_1$
and $I_3$ is contained in the zero set of $I_2+I_3$, which is $F$.
This shows the second assertion. 

If the $L$ were a linear space of dimension~3 or higher,
then the zero set of $I_1+I_3$ would be a quadratic surface or even higher
dimensional. But this zero set must be contained in $F$, so this is
not possible. This shows $\dim(L)\le 2$. 

By a suitable choice of the generators $V_1,\dots,V_5$, we can
achieve that the matrix $(V_{ij})_{i,j}$ is skew symmetric.
Then $L$ is the common zero set of $V_{5,1},\dots,V_{5,4}$.
This shows $\dim(L)\ge 1$.
\end{proof}

The first substep of constructing a point on $F$ with coordinates
in $k$ or in a quadratic extension is now easy to describe: 
the intersection of $L$ and the zero set of $F_5$ is defined over $k$,
and it is either a line, or a point, or a set of two points, 
or a plane conic (maybe reducible). 
In the first case, we take any point on the line.
In the second case, we take the point.
In the third case, we take one of the two points.
In the fourth case, we choose any line in $L$ and intersect it with
$F_5$, and we have one of the first three cases.
The generic case is that $L$ is a line intersecting $F$ in two points,
which are conjugate in a quadratic field extension.

\begin{rem}
A closer analysis shows that in the case where $L$ is a line, the
intersection is either a point or two points; and if $L$ is a plane,
then the projection $\P^5\dashrightarrow\P^2$ with center $L$ restricts to
a birational map $F\dashrightarrow\P^2$, which gives a shortcut
to the parametrization problem. We will omit the proof because it
is not necessary for the correctness proof of the algorithm.
\end{rem}

The second substep is necessary if the constructed point $q\in F$ has
coordinates in a quadratic extension $K$. Then there is also a conjugate
point $\bar{q}\in F$.

\begin{lem}
Let $q\in F$ be a point with coordinates in a quadratic field extension 
$K$ of $k$ (but not in $k$) and $\bar{q}$ its conjugate. Assume that the line $q\bar{q}$ 
is not contained in $F$. Then there exist at most $5$ 
tangent lines $T\in T_qF$ such that either $T$ and its conjugate $\bar{T}\in T_{\bar{q}}F$
are coplanar, or the $3$-plane $N$ generated by $T$ and $\bar{T}$ 
intersects $F$ in a set of positive dimension.

In all remaining cases, either $T$ is contained in $F$, and $N\cap F=T\cup\bar{T}\cup \{p\}$ where $p$ is a point in $F$ defined over $k$, or $N$ intersects $F$ in precisely $3$ points, namely $q$ and  
$\bar{q}$ (both with multiplicity $2$), and a third point $p$ which is defined over $k$. 
\end{lem}

\begin{proof}
We will show that $T_qF\cap T_{\bar{q}}F$ is either a point or the empty set. 
Let us first prove that the intersection $T_qF\cap T_{\bar{q}}F$ does not contain the line $q\bar{q}$. 
If $q\bar{q}\subseteq T_qF\cap T_{\bar{q}}F$, then 
$q\bar{q}$ is tangent to $F$ at the points $q$ and $\bar{q}$. We claim that, in such a case, the line $q\bar{q}$ would be in
$F$ which would be in contradiction with the hypothesis. Indeed, if $q\bar{q}$ is not in $F$ one can 
take $x$ and $y$ two generic points in 
$F$, then the $3$-pane generated by $q$, $\bar{q}$, $x$ and $y$ intersects $F$ 
in $q$ and $\bar{q}$ with multiplicity $2$ and in 
$x$ and $y$ with multiplicity $1$ which contradicts the degree of $F$ being $5$. Therefore, the line $q\bar{q}$ is not contained in  $T_qF\cap T_{\bar{q}}F$.

Suppose now that $T_qF\cap T_{\bar{q}}F =m\neq q\bar{q}$
where $m$ is a line in $\P^5$ defined over $k$. Then neither $q$  nor $\bar{q}$ is in $m$ (since if $q \in m$, so is $\bar{q}$ and therefore $m=q\bar{q}$). Now we define $\pi :\P^5\to \P^3$ to be the projection from the points $q$ and $\bar{q}$. Then 
$\pi\mid_F$ is the blow up of $F$ at the points $q$ and $\bar{q}$ and the Zariski closure
of $\pi (F)$ is a cubic surface $S$ in $\P^3$ (note that $S$ is a cubic surface if and only if $q\bar{q}$ is not contained in $F$). If $T_qF\cap T_{\bar{q}}F =m$ and $q$ and $\bar{q}$ are 
not in the line $m$, then both tangent planes $T_q(F)$ and 
$T_{\bar{q}}(F)$ map to the same line $\ell$ in $S$. This line would then be the exceptional divisor of two different points (or tangent directions) in $\P^2$, among the six that are blown up, which is not possible.

Now we are in the situation that $T_qF\cap T_{\bar{q}}T$ is either a point or the empty set. 
In this situation the tangent planes $T_q(F)$ and $T_{\bar{q}}(F)$ map 
to (conjugate) lines $\ell$ and $\bar{\ell}$ in $S$. If $T_q(F)$ and $T_{\bar{q}}(F)$ intersect at a point, 
then the lines $\ell$ and $\bar{\ell}$ intersect at a point $s\in S$. In such a case there is another line in $S$ 
intersecting both $\ell$ and $\bar{\ell}$. If $T_q(F)$ and $T_{\bar{q}}(F)$ do not intersect, then 
$\ell$ and $\bar{\ell}$ are disjoint and therefore there exist at most $5$ 
lines in $S$ intersecting both $\ell$ and $\bar{\ell}$ (see \cite[Lemma 1.2]{Polo_Top:08}).

Via the projection $\pi$ the plane $N$ goes to a line $n$ in $\P^3$ defined over $k$ 
that intersects both $\ell$ and $\bar{\ell}$ at two conjugate points $t=\pi(T)$ and $\bar{t}=\pi(\bar{T})$ respectively. 
When the line $n$ is a line in $S$, we are in the previous cases described above. Otherwise $n\cap S=\{t,\bar{t},x\}$ where $x$ is a point in $S$ defined over $k$. 
In this case $N\cap F=T\cup\bar{T}\cup \{p\}$ if $T\subseteq F$ or $N\cap F=\{ q,\bar{q}, p\}$ otherwise, with $p$ being a point defined over $k$ and $\pi (p)=x$.
\end{proof}

The Lemma gives a construction for a $k$-rational point on $F$ from
a point $q\in F$ defined over a quadratic extension $K$: 
let $r\ne q$ be a point with coordinates in $K$ in $T_qF$, but not
in $(T_qF\cap T_{\bar{q}}F)$. Let $\bar{r}$ be its conjugate.
Then intersect the 3-plane $N$ generated by $q,\bar{q},r,\bar{r}$
with $F$. If we are not unlucky, then the intersection contains a single point defined
over $k$. We can be unlucky at most $5$ times.   

\begin{rem}
In \cite{Sheperd-Barron:92}, Sheperd-Barron suggests to choose the line $T$ 
(or the line $qr$ in the above construction) parametrically, and compute the moving
intersection point in terms of this parameter. The parameter can be
chosen as an element in $K$, or equivalently two elements in $k$.
We implemented this method also in Magma, but the computing time
is larger than in the method we suggest in the following. 
Moreover, this method leads to a parametrization of algebraic degree~10, 
which is twice as large as the degree of the parametrization computed by
the method below.

In \cite[Exercise 3.1.4]{Hassett:09}, Hassett gives another method for constructing 
a $k$-rational point: the zero set of three generic quadrics in the ideal of $F$ 
decomposes into $F$ and a cubic rational scroll. The scroll has a unique -1-curve, which
intersects $F$ in a single $k$-rational point. However, generic choices are not free,
they increase the coefficients, and it seems not so easy to analyze the non-generic
cases for this method.

Still another method explained in \cite{Swinnerton:72} is due to Enriques: by
generic projection, one obtains an image $F'$ of $F$ in $\P^3$ with a rational quintic
double curve. It can be parametrized over $k$; compute two points on it. The line
through the two points generically intersects $F'$ in a single smooth point which
lifts back to a $k$-rational point on $F$. This method is computationally very
expensive, as we observed by testing it with a few simple examples.
\end{rem}

\paragraph{Second step.}
Given a point $p\in F$ defined over $k$, we consider
the projection map $\pi:\P^5\dashrightarrow\P^2$ with projection 
center equal to the tangent plane $T_pF$. The restriction
of $\pi$ to $F$ will be birational.

\begin{lem}
With the notation as above, the restriction $\pi|_F:F\dashrightarrow\P^2$
is birational.

If $p$ does not lie on one of the 10 lines, then the parametrization
has algebraic degree~5. If it lies on at least one line, then
the algebraic degree is smaller than 5.
\end{lem}

\begin{proof}
Two generic hyperplanes through $T_pF$ intersect in 5 points,
including the intersection at $p$. The intersection multiplicity at $p$
is equal to 4, because hyperplane sections through $T_pF$ have a
double point at $p$. Hence there is exactly one moving intersection,
and so the map $\pi|_F$ is birational.

The algebraic degree of the inverse map is equal to the number of intersections
of the image of a generic line in the parameter plane and a generic
hyperplane. In this case, the image of a generic line is also a hyperplane,
hence the algebraic degree is 5. When $p$ lies on a line, then the
line is a fixed component reduces the algebraic degree by 1. If $p$
lies on 2 lines, then the mapping degree drops by 2.
\end{proof}

\begin{rem}
In case $F$ is a minimal Del Pezzo surface, i.e. it does not have
Galois orbits of pairwise disjoint lines, then the smallest possible
algebraic degree of a parametrization is 5 (see \cite{Schicho:06b}).
So in this case the given construction has smallest possible degree.
\end{rem}

\subsection{Completeness of the Construction} \label{sec:complete}

To prove that any Del Pezzo surface $F$ of degree~5 is $k$-isomorphic
to a surface constructed by the method in section~\ref{sec:cons},
it suffices to show that $F$ has a parametrization defined by quintics
passing with multiplicity~2 through 5 points $q_1,\dots,q_5$ in general
position (recall that this only means that no three of these points 
are collinear). We will do that in this short section; another proof
of the completion is given in section~\ref{sec:recog} through the
classification.

It is clear that $F$ does contain a point $p$ with
coordinates in $k$ and is not contained in one of the 10 lines. 
Indeed, we have already seen that $F$ has a parametrization, hence
the set of all points defined over $k$ is Zariski-dense and can therefore
not be contained in the union of the 10 lines.

\begin{thm}
Assume that $p\in F$ is defined over $k$ and not contained in one of the
10 lines on $F$. Then the inverse of the birational projection $F\to\P^2$
from the tangent line $T_pF$ is a parametrization defined by quintics
passing with multiplicity~2 through 5 points in general position.
\end{thm}

\begin{proof}
We have already proven that the algebraic degree of the parametrization 
is 5, hence the parametrization is defined by quintics.
Let $q_1,\dots,q_n$ be the base points of the parametrization
(including infinitely near), and let $m_1,\dots,m_n$ be their
respective multiplicities. The number of moving intersections of
two quintics in the linear system $\Gamma$ defining the parametrization is
equal to 5, and the genus of a generic element is equal to 1, because
it is birational to a generic hyperplane section of $F$. This gives
the numeric conditions
\[ 25 - m_1^2-\dots-m_n^2 = 5 , \]
\[ 6 - \frac{m_1(m_1-1)}{2}-\dots-\frac{m_n(m_n-1)}{2} = 1 . \]
This gives only two cases.

Case~1: $n=5,m_1=\dots=m_5=2$. In this case, any line through 3 of
the base points would have at least 6 intersections, counted with
multiplicity, with any quintic in $\Gamma$, hence it would be a fixed
component, which is not possible; it follows that no three base points
are collinear. Moreover, no base points is infinitely near to
another, because this would give rise to the blow down of a -2-curve
which would give a singular point on the surface $F$, in contrast
to the assumption that $F$ is nonsingular. Hence $q_1,\dots,q_5$
are double points of $\Gamma$, and they are in generic position.

Case~2: $n=6,m_1=3,m_2=m_3=2,m_4=m_5=m_6=1$. Then the parametrization
map $\P^2\dashrightarrow \P^3$ is the product of the blowup at the 6 base points
and some birational regular map blowing down the fundamental curves.
Because $F$ is $\bar{k}$-isomorphic to the blowup of $\P^2$ at 4 points,
the number of fundamental curves is 2. The fundamental curves are the
curves which have no moving intersection points with the curves in $\Gamma$,
and these are the two lines $L(q_1,q_2)$ and $L(q_1,q_3)$. Let $p_1$
be the image of the fundamental curve $L(q_1,q_2)$. When
we compose the parametrization with the projection from the tangent plane
at $p_1$, then we get the rational map defined by all quintics in $\Gamma$
having $L(q_1,q_2)$ as a double component. Canceling the common factor,
this is the system of cubics with double point in $q_2$ and passing
through $q_1,q_4,q_5,q_6$. On the other hand, this should be the identity
map, and this is not the case. So this case does not happen.
\end{proof}

\section{Deciding Isomorphy} \label{sec:recog}

In this section, we give an algorithm for deciding whether two given
anticanonically embedded Del Pezzo surfaces $F_1$ and $F_2$ of degree~5 are isomorphic
over $k$.
The section makes use of Galois cohomology and $k$-twists: an approach which gives an
alternative (non-constructive) proof of the parametrizability of degree 5 Del Pezzo surfaces
(See \cite{Skorobogatov:01}).
Through cohomology, we see that the isomorphism class over $k$ is determined by the action of Galois on
the graph of exceptional lines and we explain how to choose the seed in the construction in order to obtain any
prescribed isomorphism class. 
\cite{Serre:94}, \cite{Skorobogatov:01} and\cite{Polo:07} are basic references here.

We write $G_k$ for $G(\bar{k}/k)$. Let $F$ denote a 
degree 5 Del Pezzo over $k$. The key point for degree 5
is that all such $F$ are still isomorphic over $\bar{k}$ (because any two sets of four
points in the projective plane with no three collinear are conjugate under $PGL_3$),
so are Galois twists of each other, but that the automorphism groups are finite, so that
the $H^1$ Galois cohomology group defining the set of twists is relatively easy to 
describe. From this, it is easy to just write down an $F$ for each cohomology class,
thus giving representatives for all of the $k$-isomorphism classes.

The Appendix to Section 3.1 of \cite{Skorobogatov:01} gives a classification of these
twists. However, the description there is rather abstract: they are given as quotients
of the set of stable points of a Grassmannian by a twisted torus. We present the
twisting theory in a more elementary fashion here, leading directly to the concrete
description of the isomorphism classes in terms of seeds with given splitting behaviour
over $k$.

\subsection{Classification of $k$-isomorphism classes}

Let $F_0$ denote the standard ``split" surface, the projective plane $\P^2_k$ blown
up at the four $k$-rational points $(1:0:0)$, $(0:1:0)$, $(0:0:1)$, $(1:1:1)$
(see Example~1).

Let $E$ denote the incidence graph for the ten exceptional lines on $F$ which are
all defined over $\bar{k}$. 
From the root system
description of $E$ in \cite{Manin:74}, we can identify $Aut(E)$ with the symmetric
group on 5 elements $S_5$ and the 10 exceptional lines with the set of pairs
$\{\{i,j\} : 1 \le i < j \le 5\}$,
so that the action of $Aut(E)$ corresponds to the natural action
of $S_5$ on the pairs.

Let $Aut(F)$ denote the group of algebraic automorphisms of $F$ over
$\bar{k}$. $Aut(F)$ naturally acts on the exceptional lines preserving
incidence relations, which leads to a homomorphism $\psi: Aut(F) \rightarrow Aut(E)$.

\begin{lem} \label{lem:autF}
$\psi$ is an isomorphism.
\end{lem}

\begin{proof}
This is Lemma 3.1.7 of \cite{Skorobogatov:01}. We give a more elementary proof here
that doesn't use the moduli space description of $F$ but instead reduces to the degree 6
Del Pezzo case. 

We identify $F$ with a blow-up of the plane at 4 points.
If $f \in Aut(F)$ fixes each exceptional line, then it comes from an automorphism of the
plane that fixes each of the
4 points. Such an automorphism is trivial if no 3 points lie on a line. Thus
$\psi$ is injective. 

To show surjectivity, it suffices to prove that $Aut(F)$
is transitive on lines and that its stabiliser of any particular line is 
$D_6$. If we blow down any exceptional line $L$, we get a non-degenerate
degree 6 Del Pezzo $F_1$ whose exceptional lines are the images of exceptional lines
of $F$ that don't intersect $L$. Any automorphism of $F_1$ that fixes the image
point $p$ of $L$ will lift to an automorphism of $F$ (that fixes $L$). Considering
the configuration $E$, it is then easy to see that it suffices to prove that the
automorphisms of $F_1$ that preserve a point $p$ not on an exceptional line induce
the full group of graph automorphisms of its exceptional lines. This is true because
the automorphism group of $F_1$ is a split extension of $D_6\cong C_2 \times D_3$ by
a 1-dimensional torus $T$ where $T$ acts transitively on the complement of the exceptional 
lines and is precisely the subgroup fixing all six of these. If $F_1$ is isomorphic to the
plane blown up at three points $P_i$,
the $T.D_3$ part comes from automorphisms of the plane preserving $\{P_i\}$ and
the extra $C_2$ comes from the plane Cremona transform based at the $P_i$.
\end{proof}

We will therefore identify $Aut(F)$ with $S_5$ through its action on the lines. Since the
action of $Aut(F)$ on these is equivariant with respect to the $G_k$ action, if all
of the exceptional lines are defined over $k$ then all elements of $Aut(F)$ are
defined over $k$ also. In particular, this holds for $F_0$ and $Aut(F_0)$ can be
identified with $S_5$ with trivial $G_k$ action.

\begin{lem} \label{lem:galois}
We have the following bijective correspondence
$$  \{\mbox{$k$-isomorphism classes of non-degenerate degree 5 Del Pezzos}\} $$
$$  \Updownarrow $$
$$  H^1(G_k,S_5) = \{\mbox{homomorphisms $G_k \rightarrow S_5$ up to
		$S_5$-conjugacy}\} $$
under which, a $k$-isomorphism class $[F]$ corresponds to the homomorphism
giving the action of $G_k$ on its graph of exceptional lines $E$.

In particular, the $k$-isomorphism class of a 5 Del Pezzo is
determined by the action of $G_k$ on its exceptional lines.
\end{lem}

\begin{proof}
The bijection (Thm. 3.1.3 of \cite{Skorobogatov:01}) comes from standard twisting theory,
identifying the set $k$-twists of a  quasi-projective variety $X$ with the elements
of $H^1(G_k,Aut_{\bar{k}}(X))$ (see Ch. 3, \S 1, \cite{Serre:94}).

We get the lower equality as $Aut(F_0)$ is equal to $S_5$ with trivial $G_k$-action.

That the cocycle class $[u_F]$ of a twist $F$ of $F_0$ corresponds to the homomorphism
of $G_k$ giving the action on its exceptional graph $E$ is an easy consequence of the
definition of $u_F$ ($u_F(\sigma) = f^{-1}\circ \sigma(f)$ where $f: F_0 \rightarrow F$
is an isomorphism over $\bar{k}$). Note that replacing
$u_F$ by an $S_5$-conjugate homomorphism just corresponds to relabelling the
elements of $E$.
\end{proof}

We now show how to construct a particular Del Pezzo $F$ that has a given
$G_k$ action $f:G_k \rightarrow S_5$ on its set of exceptional lines.
The fixed field $L$ of the
kernel of the $G_k$ action is the splitting field of the 10 lines, which
coincides with the splitting field of the seed $Q$ in case the surface
has been constructed as in section~2.
The possible isomorphism types for subgroups of $S_5$ are those occurring in
Table 1. Straightforward computation shows that:

\begin{enumerate}
\item Each isomorphism type of subgroup different from $C_2$ or $C_2\times C_2$ 
occurs uniquely up to conjugacy in
	$S_5$ (ie, $A \cong B \Rightarrow A$ is conjugate to $B$). The isomorphism type $C_2$
	occurs for two conjugacy classes: $\langle(12)\rangle$ and 
	$\langle(12)(34)\rangle$. The isomorphism type of $C_2\times C_2$ occurs also twice: 
$\langle (12),(34)\rangle$ and $\langle (12)(34),(13)(24)\rangle$. 
\item For $A \le S_5$, $Aut(A)$ is induced by $N_{S_5}(A)/C_{S_5}(A)$, except
	if $A \cong D_4$ or $A \cong D_6$ when $N_{S_5}(A)/C_{S_5}(A)$
	induces the group of inner automorphisms $Inn(A)$ with $[Aut(A):Inn(A)]=2$.
\end{enumerate}

This implies that, up to $S_5$-conjugacy, $f$ is {\it completely determined} by $L$ unless
$G(L/k)$ is isomorphic to $C_2$, $C_2\times C_2$, $D_4$ or $D_6$, when there are two
$S_5$-conjugacy classes of $f \leftrightarrow L$.

Explicitly, for $G(L/k)=D_4=\langle\sigma,\tau | 
\sigma^4=\tau^2=1\quad \tau\sigma\tau^{-1}=\sigma^{-1}\rangle$, the two classes are
$$ [\sigma \mapsto (1234), \tau \mapsto (13)]\qquad\mbox{and}\qquad
   [\sigma \mapsto (1234), \tau \mapsto (12)(34)] $$
and for $G(L/k)=D_6=\langle\sigma,\tau | \sigma^6=\tau^2=1\quad
\tau\sigma\tau^{-1}=\sigma^{-1}\rangle$, the two classes are
$$ [\sigma \mapsto (123)(45), \tau \mapsto (12)]\qquad\mbox{and}\qquad
   [\sigma \mapsto (123)(45), \tau \mapsto (12)(45)] $$

\begin{lem}
Let $Q\in k[x]$ be a quintic normed and squarefree polynomial,
with roots $P_1,\dots,P_5\in\bar{k}$.
Let $F$ be the Del Pezzo surface constructed as in section~2 with seed $Q$.
Then the homomorphism $f: G_k \rightarrow S_5$ corresponding to the 
$k$-isomorphism class of $F$ under the correspondence of Lemma~\ref{lem:galois}
is the permutation representation of $G_k$ on $\{P_1,\ldots,P_5\}$.

Conversely, for any Galois extension $L$ such that $G(L/k)$ is isomorphic to
a subgroup of $S_5$, and an $S_5$-conjugacy class of $f: G(L/k) \hookrightarrow S_5$, 
$f$ comes from an appropriate quintic normed and squarefree polynomial.
\end{lem}
 
\begin{proof}
The lines on $F$ are in bijective corresponce with the 2-element subsets
of $S := \{P_1,\ldots,P_5\}$; since we identified the action of $Aut(F)$ on $E$
with the natural action of $S_5$ on pairs, the first statement follows.

For a given $L$, we need to find $Q$ such that
the permutation Galois action on its roots is of the form $G_k \twoheadrightarrow
G(L/k) \hookrightarrow S_5$. If $\#f(G(L/k))$ is divisible by 5, then we can
take $K$ as the fixed field of $f^{-1}(S_4) \le G(L/k)$ and take $S$ as the
$G(L/k)$ orbit of a primitive element (over $k$) of $K$. An $S$ for the other
cases can be constructed similarly. In the non-$C_2,D_4,D_6$ cases, we are done.
In these three cases, we also need to show that we can take an $S$ as above that
leads to either of the 2 conjugacy classes of embedding
$G(L/k) \hookrightarrow S_5$. 

In case $G(L/k) = C_2$, pick $P_1,P_2$ conjugate over $L$ and 
$P_3,P_4,P_5 \in k$ for $\langle(12)\rangle$; pick $P_1,P_2$ and $P_3,P_4$
conjugate over $L$ and $P_5 \in k$ for $\langle(12)(34)\rangle$.

In case $G(L/k)=C_2\times C_2$, then we have that $L=k(\alpha )$ where $\alpha$ is a primitive element of $L$ 
and $L=L_1\otimes_kL_2$ where $L_1$ and $L_2$ are two distinct quadratic extensions of $k$ contained in $L$. 
If $P_1$, $P_2$, $P_3$ and $P_4$ are the roots of the minimal polynomials of $\alpha$ and $P_5$ is in $k$ we have 
the embedding of $G(L/k)$ in $S_5$ as the Klein group $V_4=\langle (12)(34),(13)(24)\rangle$. If $P_1$ and $P_2$ are conjugate elements in $L_1$, 
$P_3$ and $P_4$ are conjugate elements in $L_2$ and $P_5$ is in $k$, then the embedding of $G(L/k)$ on $S_5$ is 
the group $\langle (12), (34)\rangle$.

Assume $G(L/k) = D_4 = \langle\sigma,\tau | 
\sigma^4=\tau^2=1\quad \tau\sigma\tau^{-1}=\sigma^{-1}\rangle$. Let 
$K_1 = L^{\langle\tau\rangle}$, $K = L^{\langle\sigma\tau\rangle}$.
$[K_1:k]=4$ and $L$ is the Galois closure of
$K_1$ in $\bar{k}$. Similarly for $K$. We pick $P_5 \in k$ and 
$P_1,\ldots,P_4$ the conjugates of a primitive element of $K/k$ for
one conjugacy class of $f$ and $P_1,\ldots,P_4$ the conjugates of a
primitive element of $K_1/k$ for the other.

The case $G(L/k) = D_6$ is similar.
\end{proof}

In particular, the surfaces constructed cover every $k$-isomorphism class and
are plainly parametrizable over $k$ (they are all constructed by blowing up a Galois-stable set of
5 points on a plane conic and then blowing down the strict transform of the conic).
Thus, we can deduce that all non-degenerate degree
5 Del Pezzos over $k$ are parametrizable over $k$ from this classification without the explicit
construction of $k$-rational points in Section~3. The Lemma also gives an alternative proof
of the completeness result of Section~4.

\begin{table}[h]
\begin{center}
\begin{tabular}{|c|r|c|}
\hline
   Isomorphism type & Polynomial defined by the five points   & Parametric  \\
 of Galois Group, & in $\overline{\Q}$ that determine the  & degree
    \\
 number of orbits & Del Pezzo surface of degree $5$&\  \\

\hline
 $S_5,\ 1$ & $x^5 - 2x^4 - 3x^3 + 6x^2 - 1$ & 5 \\
 $A_5,\ 1$ & $x^5 - 11x^3 - 5x^2 + 18x + 9$ & 5\\
 $S_4,\ 2$ & $(x^4 - 4x^2 - x + 1)x$ & 3\\
 $H_{20},\ 1$ & $x^5 - 9x^3 - 4x^2 + 17x + 12$ & 5 \\
 $A_4,\ 2$ & $(x^4 - x^3 - 7x^2 + 2x + 9)x$ & 3\\
 $D_6,\ 3$  & $(x^3-2)(x^2-5)$ & 4\\
 $D_6,\ 3$ & $(x^3+2)(x^2+x+1)$ & 4\\
 $D_5,\ 2$  & $x^5 - x^4 - 5x^3 + 4x^2 + 3x - 1$ & 5 \\
 $D_4,\ 3$ & $(x^4 - 4x^2 + 5)x$ & 3\\
 $D_4,\ 3$ & $(x^4 - 8x^2 - 4)x$ & 3\\
 $S_3,\ 4$ & $(x^3 - x^2 - 3x + 1)(x+1)x$ & 3 \\
 $C_6,\ 3$ & $(x^3 - x^2 - 2x + 1)(x^2+1)$ & 4\\
 $C_5,\ 2$ & $x^5 - x^4 - 4x^3 + 3x^2 + 3x - 1$ & 5 \\
 $C_4,\ 3$ & $(x^4 - x^3 - 4x^2 + 4x + 1)x$ & 3\\
 $C_2\times C_2,\ 4$ & $(x^4 - 2x^2+9)x$ & 3\\
 $C_2\times C_2,\ 5$ & $(x^2+1)(x^2-2)x$ & 3\\
 $C_3,\ 4$ & $(x^3 - x^2 - 2x + 1)(x+1)x$ & 3\\
 $C_2,\ 6$ & $(x^2+1)(x^2+4)x$ & 3\\
 $C_2,\ 7$ & $(x^2+1)(x+1)(x-1)x$ & 3\\
 $1,\ 10$ &  $(x+2)(x-2)(x+1)(x-1)x$ & 3\\
\hline
\end{tabular}
\end{center}
\caption{Seeds for constructing example Del Pezzo surfaces over $\Q$ of prescribed isomorphism type
	and number of line orbits}
\label{tab:polynomials}
\end{table}

Table \ref{tab:polynomials} below gives examples of seeds for every isomorphism type.
In the table the isomorphism type of the Galois group acting on the Del Pezzo surface and 
the number of orbits of the induced action on the 10 lines appear together with a 
polynomial $f(x)$ of degree $5$ that defines the surface. The groups $C_n$, $D_n$, $S_n$ and $A_n$ denote 
the cyclic, the dihedral, the symmetric and the alternating groups respectively and 
$H_{20}$ denotes the unique subgroup up conjugation of order $20$ of $S_5$. In case there are two conjugacy classes 
of embeddings of the group in $S_5$, a polynomial for each case is given (with the same splitting field $L$). 
The Del Pezzo surface can be obtained 
by mapping $\P^2$ into $\P^5$ by the space of quintics passing with multiplicity two through the points 
$(P_1^2:P_1:1)$, $(P_2^2:P_2:1)$, $(P_3^2:P_3:1)$, $(P_4^2:P_4:1)$, $(P_5^2:P_5:1)$, where $P_1, \ldots ,P_5$ are 
the roots of $f(x)$. In the last colum of the table we include the parametric degree of the Del Pezzo surface. 
This degree is 3 if and only if $F$ is the blowup of $\P^2$ at a Galois-invariant quadruple of points,
if and only if there exists a Galois orbit of 4 pairwise disjoint lines, if and only if $f$ has a linear factor.
The degree is 4 if and only if $F$ is the blowup of a nonsingular quadric 
at a Galois-invariant triple of points but not of $\P^2$, if and only if there exists a Galois orbit of 3 
pairwise disjoint lines but no Galois orbit of 4 pairwise disjoint lines, 
if and only if $f$ has a quadratic but no linear factor.
It is 5 if and only if $F$ is $k$-minimal, if and only if $f$ is irreducible.

\subsection{Testing for $k$-isomorphy}

Given two quintic Del Pezzo surfaces $F_1,F_2$ by two sets of quadric generators of their ideals,
we decide whether they are $k$-isomorphic. 

\paragraph{First step.}
We reduce to deciding whether the two Galois actions
on the roots of two given quintic squarefree polynomials $Q_1,Q_2$ (the seeds) are
$S_5$-conjugate.
The polynomials can be constructed by first calculating a parametrization
by quintics as in Theorem~12, and then construct the polynomial as in Lemma~4.

\paragraph{Second step.}
Given $(Q_1,Q_2)$, we first decide whether the splitting fields
coincide. Theoretically, this can be done by factoring $Q_2$ 
over the splitting field of $Q_1$ and factoring $Q_1$ over the splitting field 
of $Q_2$ (but this is, of course, not the fastest method). 
If the splitting fields
do not coincide, then the actions have different kernels and are not conjugate.
Otherwise, let $L$ be the common splitting field of the $Q_i$.
If the Galois group $G(L/k)$ is not isomorphic to $C_2$, $C_2\times C_2$, $D_4$, or $D_6$,
then the actions are conjugate by the preceding section. Otherwise we get: 

\begin{description}
\item[{\it Case i)}] Assume that $G(L/k)=C_2$. Then the actions are conjugate if and only if
the number of irreducible factors of $Q_1$ and $Q_2$ coincide -- they have either
4 factors of degree 2,1,1,1, or 3 factors of degree 2,2,1.

\item[{\it Case ii)}] Assume that $G(L/k)=C_2\times C_2$. Then the actions are conjugate if and only if
the number of irreducible factors of $Q_1$ and $Q_2$ coincide -- they have either
3 factors of degree 2,2,1, or 2 factors of degree 4,1.

\item[{\it Case iii)}] Assume that $G(L/k)=D_4$. Then both $Q_1$ and $Q_2$ have two irreducible
factors, of degree 1 and 4. Both quartic factors define non-Galois degree 4 extensions of $k$ 
which contain a unique quadratic subextension over $k$. The actions are conjugate if and only if the 
two quadratic extensions coincide. In fact, the Galois theory shows that for two non-Galois, non-conjugate
degree 4 subfields of $L$, $L_1$ and $L_2$, the quadratic subfield of $L_1$ is generated by the
square-root of the discriminant of a defining polynomial for $L_2$ and vice-versa. So we find that
the actions are conjugate if and only if the quotient of the discriminants of the quartic factors
is a square in $k$.

\item[{\it Case iv)}] Assume that $G(L/k)\cong D_6$. Then both $Q_1$ and $Q_2$ have two irreducible
factors, of degree 2 and 3. The actions are conjugate if the two splitting fields
coincide.
\end{description}

\begin{rem} If $k$ is non-perfect, everything above works with
$k^{sep}$ replacing $\bar{k}$ as the exceptional lines for $F$ are
all defined over $k^{sep}$. This follows from the fact that all smooth rational
surfaces are separably split \cite{Coombes:88}.
%
\end{rem}

\end{document}